\documentclass[final, 11pt]{article}

\def\showkeys{0}
\def\showdraftbox{0}

\usepackage{bm}
\usepackage{bbm}

\usepackage{amsthm}

\usepackage{xspace,xcolor,enumerate}
\usepackage{amsmath,amsfonts,amssymb}
\usepackage{color,graphicx}

\ifnum\showkeys=1
\usepackage[color]{showkeys}
\fi

\definecolor{darkred}{rgb}{0.5,0,0}
\definecolor{darkgreen}{rgb}{0,0.5,0}
\definecolor{darkblue}{rgb}{0,0,0.5}

\usepackage[pdfstartview=FitH,pdfpagemode=None,colorlinks,linkcolor=darkred,filecolor=blue,citecolor=darkred,urlcolor=darkred,pagebackref]{hyperref}

\usepackage{dsfont}

\usepackage[capitalise,nameinlink]{cleveref}

\ifnum\showdraftbox=1

\else

\fi

\newtheorem{theorem}{Theorem}[section]

\crefname{conjecture}{Conjecture}{Conjectures}

\crefname{definition}{Definition}{Definitions}
\newtheorem{lemma}[theorem]{Lemma}
\crefname{lemma}{Lemma}{Lemmas}

\crefname{remark}{Remark}{Remarks}
\newtheorem{proposition}[theorem]{Proposition}
\crefname{proposition}{Proposition}{Propositions}
\newtheorem{corollary}[theorem]{Corollary}
\crefname{corollary}{Corollary}{Corollaries}

\crefname{obs}{Observation}{Observations}
\newtheorem{claim}[theorem]{Claim}
\crefname{claim}{Claim}{Claims}

\crefname{fact}{Fact}{Facts}

\crefname{openprob}{Open Problem}{Open Problems}

\crefname{remk}{Remark}{Remarks}

\crefname{example}{Example}{Examples}

\def\epsilon{\varepsilon}

\def\phi{\varphi}

\newcommand{\R}{{\mathbb R}}

\usepackage{nicefrac}

\newfont{\inhead}{eufm10 scaled\magstep1}

\usepackage[top=1in,bottom=1in,left=1.25in,right=1.25in]{geometry}
\usepackage{color,graphicx}

\ifnum\showkeys=1
\usepackage[color]{showkeys}
\fi

\definecolor{darkred}{rgb}{0.5,0,0}
\definecolor{darkgreen}{rgb}{0,0.5,0}

\newcommand{\cube}{\{0,1\}^n}

\newcommand{\Hit}{h}

\title{An Asymptotically Tight Bound on the Number of Relevant Variables in a Bounded Degree Boolean Function} 

\author{John Chiarelli\footnote{{\tt jlc450@math.rutgers.edu}, Department of Mathematics, Rutgers University}\and Pooya Hatami\footnote{{\tt pooyahat@gmail.com}, Department of Computer Science, University of Texas at Austin, supported by a Simons Investigator Award (\#409864, David Zuckerman)} \and Michael Saks\footnote{{\tt saks@math.rutgers.edu}, Department of Mathematics, Rutgers University, supported in part by the Simons Foundation under award 332622}}

\begin{document}

\maketitle

\begin{abstract}
We prove that there is a constant $C\leq 6.614$ such that every Boolean function of degree at most $d$ (as a polynomial over $\mathbb{R}$)
is a $C\cdot 2^d$-junta, i.e. it depends on at most $C\cdot 2^d$ variables. This improves 
the $d\cdot 2^{d-1}$ upper bound of Nisan and Szegedy [Computational Complexity 4 (1994)]. 

The bound of $C\cdot 2^d$ is tight up to the constant $C$ as a lower bound of $2^d-1$ is achieved by a read-once decision tree of depth $d$. 
 We slightly improve the lower bound by constructing, for each positive integer $d$, a function
of degree $d$ with $3\cdot 2^{d-1}-2$ relevant variables. A similar construction was independently observed by Shinkar and Tal.
\end{abstract}

\section{Introduction}
The \emph{degree} of a  Boolean function $f:\{0,1\}^n\rightarrow\{0,1\}$, denoted $\deg(f)$, is the minimum degree of  a polynomial in $\R[x_1,...,x_n]$ that agrees with $f$ on all  inputs from $\{0,1\}^n$. (It is well known that every Boolean function
has a unique representation over the reals, called the {\em multilinear representation}, of the form
$\sum_{S \subseteq [n]} a_S\prod_{i \in S}x_i$, and that $\deg(f)$ is the degree of the multilinear
representation of $f$.) 
Minsky and Papert~\cite{MP88} initiated the study of combinatorial and computational
properties of Boolean functions based on their representation by polynomials.   
We refer the reader to the excellent book of O'Donnell~\cite{Ryan14} on analysis of Boolean functions, and surveys \cite{BuhrmanD02, hatami2011variations} discussing relations between various complexity measures of Boolean functions. 


An input variable $x_i$  is \emph{relevant} for a Boolean function $f$ if it appears in a monomial of the multilinear representation of $f$
with nonzero coefficient. Let $R(f)$ denote the number of relevant
variables of $f$. We say that $f$ is a {\em $t$-junta} if $R(f) \leq t$. Nisan and Szegedy~\cite{NS94}, proved  that $R(f)$ is at most
at most $\deg(f)\cdot 2^{\deg(f)-1}$.

Let $R_d$ denote the maximum of $R(f)$ over Boolean functions $f$ of degree at most $d$, and let $C_d=R_d2^{-d}$.
By the result of Nisan and Szegedy, $C_d \leq d/2$.
On the other hand,  $R_d\geq 2R_{d-1}+1$, since if $f$ is a degree $d-1$ Boolean function
with $R_{d-1}$ relevant variables, and $g$ is a copy of $f$ on disjoint variables, and $z$ is a new variable
then  $zf+(1-z)g$ is a degree $d$ Boolean function with $2R_{d-1}+1$ relevant variables.   Thus  $C_d \geq C_{d-1}+2^{-d}$,
and so $C_d \geq 1-2^{-d}$.  Since
$C^d$ is an increasing function of $d$ it approaches a (possibly infinite) limit $C^* \geq 1$.  

In this paper we prove:

\begin{theorem}\label{thm:main} 
There is a positive constant $C$ so that $R(f)2^{-\deg(f)}\leq C$ for all Boolean functions $f$,
and thus $C_d \leq C$ for all $d \geq 0$.  In particular $C^*$ is finite.
\end{theorem}

Throughout this paper we use $[n]=\{1,\ldots,n\}$ for the index set of the variables to Boolean function $f$.
A {\em maxonomial} of $f$ is a set  $S \subseteq [n]$ of size $\deg(f)$ for which $\prod_{i \in S}x_i$ has a nonzero coefficient
in the multilinear representation of $f$. A {\em maxonomial hitting set} is a subset $H \subseteq [n]$ that intersects
every maxonomial.    Let $h(f)$ denote the minimum size of a maxonomial hitting set for $f$
and let $\Hit_d$ denote the maximum of $h(f)$ over Boolean functions of degree $d$.
Our key lemma, proved in \cref{sec:key lemma} is:

\begin{lemma}\label{lemma:main}
$$C_d- C_{d-1} \leq \Hit_d 2^{-d},$$
\end{lemma}
which immediately implies
$C_d \leq \sum_{i=1}^d \Hit_i 2^{-i}$.

We also have:

\begin{lemma}
\label{lemma:h_d}
For any boolean function $f$, $h(f) \leq d(f)^3$ and so
for all $i \geq 1$ $\Hit_i \leq  i^3$. 
\end{lemma}

This is proved in  ~\cref{sec:h_d}.  As explained there the $h(f) \leq 2d(f)^3$ is implicit in previous work,
and an additional argument eliminates the factor of 2.

Using ~\cref{lemma:h_d}, the summation in the upper bound of ~\cref{lemma:main} converges and
\cref{thm:main} follows.

Once we establish that $C^*$ is finite, it is interesting to obtain upper and lower bounds on $C^*$.
The best bounds we know are $3/2 \leq C^* \leq 6.614$.  We discuss these bounds in
\cref{sec:bounds}.  

Filmus and Ihringer~\cite{FI18} recently considered an analog of the parameter $R(f)$ for the family of 
{\em level $k$ slice functions} which are Boolean functions whose domain is restricted to the
set of inputs of Hamming weight exactly $k$.  They showed that, provided that $\min(k,n-k)$ is sufficiently
large (at least
$B^d$ for some fixed constant $B$) then every level $k$ slice function on $n$-variables of degree at most $d$ depends on at most
$R_d$ variables.  (See \cite{FI18} for precise definitions and details.) Thus our improved upper bound on $R_d$ applies also to
the number of relevant variables of slice functions.

\section{Proof of \cref{lemma:main}}
\label{sec:key lemma}
Similar to Nisan and Szegedy, we upper bound $R(f)$ 
by assigning a weight to each variable, and bounding the total weight of all variables. 
The weight of a variable used by Nisan and Szegedy is its \emph{influence} on $f$; the novelty of our approach is to use a different weight function.

For a variable $x_i$, let $\deg_i(f)$
be the maximum degree among all monomials that contain $x_i$ and have nonzero coefficient  in the multilinear representation of $f$. Let $w_i(f):=2^{-\deg_i(f)}$.
The weight of $f$, $W(f)$ is $\sum_i w_i(f)$ and  
$W_d$ denotes the maximum of $W(f)$ over all Boolean functions $f$ of degree at most $d$.

\begin{proposition}
\label{prop:W_d}
For all $d \geq 1$, $C_d=W_d$.
\end{proposition}

\begin{proof}
For any function $f$ of degree $d$, we have $W(f) \geq R(f)2^{-d}$.  Thus $W_d \geq C_d$.

To prove the reverse inequality, let $f$ be a function of degree at most $d$ with $R(f)$
as large as possible subject to $W(f)=W_d$.   We claim that $\deg_i(f)=d$ for all relevant variables.
Suppose for contradiction $\deg_i(f)<d$ for some $x_i$. Let
$g$ be the function obtained by replacing $x_i$ in $f$ by the AND of two new variables
$y_i \wedge z_i$. Then $\deg(g)=\deg(f) \leq d$ and  $W(g)=W(f)$ and $R(g)=R(f)+1$, contradicting
the choice of $f$. Since $\deg_i(f) = d$ for all $i$, we have $W(f)=R(f)2^{-d}=C_d$.
\end{proof}

Therefore to prove \cref{lemma:main} it suffices to prove that $W_d- \Hit_d 2^{-d} \leq W_{d-1}$.

Let $H$ be a maxonomial hitting set for $f$ of minimum size.  Note that 
$\deg_i(f)=d$ for all $i \in H$ (otherwise $H-\{i\}$ is a smaller maxonomial hitting set).
We have: 

\begin{equation}
\label{eqn:W}
W(f)=\sum_i w_i(f)=2^{-d}|H| + \sum_{i \not\in H}w_i(f).
\end{equation}

A {\em  partial assignment} is a mapping 
$\alpha:[n]\longrightarrow \{0,1,*\}$, and $Fixed(\alpha)$ is the set $\{i:\alpha(i) \in \{0,1\}\}$.
For $J \subseteq [n]$, $PA(J)$ is the set of partial assignments $\alpha$ with $Fixed(\alpha)=J$.
The {\em restriction} of $f$ by $\alpha$,  $f_{\alpha}$, is the function on variable set $\{x_i:i \in [n]-Fixed(\alpha)\}$ obtained by
setting $x_i=\alpha_i$ for each $i \in Fixed(\alpha)$. 

\begin{claim}
\label{claim:w_i}
Let $J \subseteq [n]$.  For any $i \not\in J$. 

$$
w_i(f) \leq 2^{-|J|} \sum_{\alpha \in PA(J)} w_i(f_{\alpha})
$$ 
\end{claim}

\begin{proof}
  Let $j \in J$ and
 write $f_0$ for the restriction of $f$ by $x_j=0$
and $f_1$ be the restriction of $f$ by $x_j=1$.  
Then $f=(1-x_j)f_0+x_jf_1$. 

We proceed by induction on $|J|$. For the basis case $|J|=1$ we have $J=\{j\}$

\begin{itemize}
\item If $f_0$ does not depend on  $x_i$, then $w_i(f)=w_i(f_1)/2$.
\item If$f_1$ does not depend on $x_i$, then $w_i(f)=w_i(f_0)/2$.
\item Suppose $f_1$ and $f_0$ both depend on $x_i$.  If $\deg(f_0) \neq \deg(f_1)$ then 
$\deg_i(f)=1+\max(\deg_i(f_0),\deg_i(f_1))$ and so $w_i(f) = \frac{1}{2}\min(w_i(f_0),w_i(f_1))$.  If
$\deg(f_0)=\deg(f_1)$ then every monomial of $f_0$ appears in $f=x_j(f_1-f_0)+f_0$ with the same 
coefficient and therefore $w_i(f) \leq w_i(f_0)=\frac{1}{2}(w_i(f_0)+w_i(f_1))$.
\end{itemize}

In every case, we have $w_i(f) \leq (w_i(f_0)+w_i(f_1))/2$, as required.

For the induction step, assume $|J| \geq 2$.   By the case $|J|=1$,
we have $w_i(f) \leq \frac{1}{2}(w_i(f_0)+w_i(f_1))$.  Apply the induction hypothesis
separately to $f_0$ and $f_1$ with the set of variables $J-\{j\}$:

\begin{eqnarray*}
w_i(f) & \leq & \frac{1}{2}(w_i(f_0)+w_i(f_1)) \\
& \leq & \frac{1}{2}\left(\sum_{\beta \in PA(J-\{j\})} w_i(f_{0,\beta})+\sum_{\beta \in PA(J-\{j\})} w_i(f_{1,\beta}) \right)\\
& \leq &\sum_{\alpha \in PA(J)} w_i(f_{\alpha}).
\end{eqnarray*}

\end{proof}

To complete the proof of \cref{lemma:main}, apply \cref{claim:w_i} with $J$ being the minimum size hitting set $H$, and sum over all $i \in [n]-H$ to get:

\[
\sum_{i \in [n]-H}w_i(f) \leq 2^{-|H|} \sum_{i \in [n]-H} \sum_{\alpha \in PA(J)}w_i(f_{\alpha}) \leq
2^{-|H|} \sum_{\alpha \in PA(H)}w(f_{\alpha}) \leq W_{d-1},
\]
since $\deg(f_{\alpha}) \leq d-1$ for all $\alpha \in PA(H)$.

Combining with (\cref{eqn:W}) gives $W_d \leq W_{d-1} + |H|\cdot 2^{-d}$ as required to prove \cref{lemma:main}.

\section{Bounds on $C^*$}
\label{sec:bounds}

\cref{lemma:main} implies that $C_d \leq \sum_{i=1}^{d}2^{-i} \Hit_i$.  Combined with ~\cref{lemma:h_d} gives $C_d \leq \sum_{i=1}^d i^32^{-i}$
and the limiting value $C^{*} \leq \sum_{i=1}^{\infty} i^3 2^{-i}$.  This sum is equal to 26 (which can be shown,
for example, by using $\sum_{i \geq 1} \binom{i}{j}2^{-i} = 2$ for all $j$,
and $i^3=6\binom{i}{3}+6\binom{i}{2}+1$) and thus $C^* \leq 26$.   
As noted in the introduction, $R_d \geq 2^d-1$ and so $C^* \geq 1$.

The best upper and lower bounds we know on $C^*$ are:

\begin{theorem}
\label{thm:C*}
$3/2 \leq C^* \leq 6.614$.
\end{theorem}

For the upper bound, \cref{lemma:main} implies that for any positive integer $d$,

$$C^* \leq C_d + \sum_{i=d+1}^\infty 2^{-i}h_i.$$ 

Since $C_d \leq d/2$ by the result of Nisan and Szegedy mentioned in the introduction, we have

$$C^* \leq \min_{d} \left( \frac{d}{2}+\sum_{i=d+1}^{\infty} i^32^{-i}\right).$$

The minimum occurs at the largest $d$ for which the summand $d^3 2^{-d}>1/2$ which is 12.
Evaluating the right hand side for $d=12$ gives $C^* \leq 6.614$.

We lower bound $C^*$ by exhibiting, 
for each $d$ a function  $\Xi_d$ of degree $d$ with $l(d)=\frac{3}{2}2^d -2$ relevant variables.
(A similar construction was found independently by Shinkar and Tal \cite{Tal17}.
It is more convenient to switch our Boolean set to be $\{-1,1\}$.

We define $\Xi_d:\{-1,1\}^{l(d)}\rightarrow\{-1,1\}$ as follows. $\Xi_1:\{-1,1\}\rightarrow\{-1,1\}$ is the identity function and for all $d>1$, 
$\Xi_d$ on $l(d)=2l(d-1)+2$ variables is defined in terms of $\Xi_{d-1}$ as follows: $$\Xi_d(s,t,\vec{x},\vec{y})=\frac{s+t}{2}\Xi_{d-1}(\vec{x})+\frac{s-t}{2}
\Xi_{d-1}(\vec{y})$$ for all $s,t\in\{-1,1\}$ and $\vec{x},\vec{y}\in\{-1,1\}^{l(d-1)}$.  It is evident
from the definition that $\deg(\Xi_d)=1+\deg(\Xi_{d-1})$ which is $d$ by induction (as for the base case 
$d=1$, $\Xi_1$ is linear).  It is easily checked that $\Xi_{d}$ depends on all of its variables and that
$\Xi_d(s,t,\vec{x},\vec{y})$ equals
$s*\Xi_{d-1}(\vec{x})$ if $s=t$ and equals $s*\Xi_{d-1}(\vec(y))$ if $s \neq t$, and is therefore Boolean.


\section{Proof of \cref{lemma:h_d}}
\label{sec:h_d}

In this section, we will show that for any Boolean function $f$, $h(f)\le d(f)^3$.

In an unpublished argument, Nisan and Smolensky (see Lemma 6 of \cite{BuhrmanD02}) proved  $\Hit_i \leq d(f)bs(f)$, where $bs(f)$ is the block 
sensitivity of $f$.  A result from \cite{NS94} (see Theorem 4 of \cite{BuhrmanD02}) says $bs(f)\leq 2d(f)^2$, 
which implies that $h(f) \leq 2d(f)^3$.  We now show how to eliminate the factor 2 in the upper bound.

Recall that for Boolean functions $f:\{0,1\}^n\rightarrow\{0,1\}$ and $g:\{0,1\}^m\rightarrow\{0,1\}$, their
{\em composition} $$f\circ g=f(g(t_{1,1},t_{1,2},\ldots),g(t_{2,1},t_{2,2},\ldots),\ldots)$$ is a Boolean function in $mn$ variables with variable set $\{t_{i,j}:i\in[n], j\in[m]\}$.  
We begin by showing that degree and maxonomial hitting set size are multiplicative, i.e.,  $d(f\circ g)=d(f)d(g)$ and 
$h(f\circ g)=h(f)h(g)$.  The former property is well known:  the set of monomials of  $f\circ g$
is the set of all monomials of the form $c_M\prod_{x_i\in M} m_i$, where $M=c_M\prod_{x_i\in M}x_i$ is a monomial of $f(x_1,x_2,\ldots)$ and, for all 
relevant $i$, $m_i$ is a monomial of $g(t_{i,1},t_{i,2},\ldots)$.  The degree of such a monomial is maximized when $M$ and all corresponding $m_i$'s are 
maxonomials, in which case its degree is $\sum_{x_i\in M} d(g)=d(f)d(g).$

We now show that $h(f\circ g)=h(f)h(g)$.  It is easy to check that $S_0=\{(i,j):i\in S_1, j\in S_2\}$ is a maxonomial hitting set of $f\circ g$, where $S_1$ is any maxonomial 
hitting set of $f(x_1,x_2,\ldots)$ and $S_2$ is any maxonomial hitting set of $g(t_{1,1},t_{1,2},\ldots)$; therefore, $h(f\circ g)\le h(f)h(g)$.

We now show that $h(f\circ g)\geq h(f)h(g)$.  Let $S\subseteq\{(i,j): i\in [n], j\in [m]\}$ be a maxonomial hitting set of $f\circ g$.  Let $S_i$ be the set of pairs in $S$ with first coordinate $i$, and let $S'$ be the set of all $i\in [n]$ such that $S_i$ 
is a maxonomial hitting set of $g(t_{i,1},t_{i,2},\ldots)$.  We claim that $S'$ is a maxonomial hitting set of $f(x_1,x_2,
\ldots)$. (Suppose not. Then there is a maxonomial $M_f$ that $S'$ does not cover.  
For each $i$ such that $x_i\in M_f$, there is a maxonomial $M_i$ of $g(t_{i,1},t_{i,2},\ldots)$ that is not hit by $S_i$.  Then, $\prod_{i:x_i\in M_f}M_i$ is a maxonomial of $f\circ g$ that is not hit by $S$.) This implies $|S'|\geq h(f)$. Since  $i \in S'$ implies $S_i\geq h(g)$,  $|S|\geq h(f)h(g)$. Therefore  $h(f\circ g)\geq h(f)h(g)$, and so $h(f\circ g)=h(f)h(g)$.

Returning to the proof of the main result, assume for the sake of contradiction, that there exists a degree $d$ 
Boolean $f$ with  $h(f)\geq d^3+1$.  We have $d>1$, since the only functions with degree $d \leq 1$, the constant and univariate functions, have maxonomial hitting 
set size $d$. Consider the function $F:=\circ^{d^3}f$, the composition of $f$ with itself $d^3$ times.  
By the multiplicative property of Boolean degree and hitting set size, we have $d(F)=d^{d^3}$ and 
$h(F)=(d^3+1)^{d^3}>(d^3)^{d^3}(1+\frac{1}{d^3})^{d^3}> 2(d^3)^{(d^3)}=2d(F)^3$,
contradicting  $h(F)\leq 2d(F)^3$.  Therefore  no such $f$ exists and $h(f)\leq d(f)^3$ for all Boolean $f$.

\section*{Acknowledgements}
We thank Avishay Tal for helpful discussions, and  for sharing his python code for exhaustive search of Boolean functions on few variables. We also thank Yuval Filmus for pointing out the implications for Boolean functions on the slice.  

\bibliographystyle{amsalpha}
\bibliography{bibs}

\end{document}